\documentclass[a4paper,leqno,12pt]{amsart}
 \usepackage[usenames]{color}
 \usepackage{amssymb,amsthm,amsmath,amscd}
\usepackage{ascmac}
 \usepackage{fancybox}
 \usepackage[dvips]{graphicx}
 \usepackage{ulem}
 \usepackage[all]{xy}
 \input amssym.def
 \input amssym.tex
\newcommand{\DEG}{\mathrm{deg}} 
  \newcommand{\Fix}{\mathrm{Fix}}
 \newcommand{\Pic}{\mathrm{Pic}}
 \newcommand{\Alb}{\mathrm{Alb}}
 \newcommand{\Ord}{\mathrm{ord}}
 
\newcommand{\Dim}{\mathrm{dim}}
 \newcommand{\Ker}{\mathrm{Ker}}
 \newcommand{\Image}{\mathrm{Im}}

 \newcommand{\Z}{\mathbb{Z}}
 \newcommand{\PROJ}{\mathbb{P}}

 \newcommand{\D}{\mathfrak{d}}
 \newcommand{\RH}{\widehat{R}}
 \newcommand{\RT}{\widetilde{R}}
 \newcommand{\DT}{\widetilde{\mathfrak{d}}}
 \newcommand{\PHI}{\varphi}
 \newcommand{\ST}{\widetilde{S}}
 \newcommand{\WH}{\widehat{W}} 
 \newcommand{\WT}{\widetilde{W}} 
 \newcommand{\SigmaT}{\tilde{\sigma}}
 \newcommand{\FDT}{\tilde{f'}}
 \newcommand{\WC}{\check{W}} 
 \newcommand{\ALB}{\bar{\alpha}}
 \newcommand{\ALC}{\check{\alpha}}
 \newcommand{\AL}{\alpha} 
 \newcommand{\GA}{\Gamma}
  \newcommand{\GAT}{\widetilde{\Gamma}}
 
\newtheorem{prop}{Proposition}[section]
 \newtheorem{thm}[prop]{Theorem}
 \newtheorem{lem}[prop]{Lemma}
 \newtheorem{cor}[prop]{Corollary}

\theoremstyle{remark}
\newtheorem{define}[prop]{Definition}
\newtheorem{rmk}[prop]{Remark}

\numberwithin{equation}{section} 
\begin{document}
 \title{Slope inequarities for irregular cyclic covering fibrations}
\author{Hiroto Akaike}
 
\maketitle{}

 \begin{abstract}
Let $f:S\to B$ be a finite cyclic covering fibration of a fibered surface.
We study the lower bound of slope $\lambda_{f}$ when the relative irregularity $q_{f}$ is positive. 
 \end{abstract}
 %
\section*{Introduction}
Let $f:S\to B$ be a surjective morphism from a smooth projective surface $S$ 
to a smooth projective curve $B$ with connected fibers. 
We call it a fibration of genus $g$ when a general fiber is a curve of genus $g$.
A fibration is called relatively minimal, when any $(-1)$-curve is not contained in fibers.
Here we call a smooth rational curve $C$ with $C^2=-n$ a $(-n)$-curve.
A fibration is called {\it smooth} when all fibers are smooth, {\it isotrivial} when all of the 
smooth fibers are isomorphic, {\it locally trivial} when it is smooth and isotrivial.

Assume that $f:S\to B$ is a relatively minimal fibration of genus $g \geq 2$.
We denote by $K_{f}=K_{S}-f^{\ast}K_{B}$ a relative canonical divisor.
We associate three relative invariants with $f$:
\begin{align*}
 &K_{f}^2=K_{S}^2-8(g-1)(b-1),\\
 &\chi_{f}:=\chi(\mathcal{O}_{S})-(g-1)(b-1),\\
 &e_{f}:=e(S)-4(g-1)(b-1),
\end{align*} 
where $b$ and $e(S)$ respectively denote the genus of the base curve $B$ and the topological Euler-Poincar\'{e} characteristic of $S$.
Then the following are well-known:

\begin{itemize}
  \item (Noether) \quad $12\chi_{f}=K_{f}^2+e_{f}$.
  \item (Arakelov)\quad $K_{f}$ is nef.
  \item (Ueno)\quad $\chi_{f}\geq0$ and $\chi_{f}=0$ if and only if $f$ is locally trivial.
  \item  (Segre)\quad $e_{f}\geq0$ and $e_{f}=0$ if and only if $f$ is smooth.
\end{itemize}

\noindent
When $f$ is not locally trivial, we put 
$$
\lambda_{f}:=\frac{K_{f}^2}{\chi_{f}}
$$
and call it the {\it slope} of $f$ according to \cite{Xiao2}, in which Xiao succeeded in giving its effective lower bound as
$$
\lambda_{f} \geq \frac{4(g-1)}{g}.
$$
Another invariant we are interested in is the {\it relative irregularity} of $f$ defined by $q_{f}:=q(S)-b$,
where $q(S):=\Dim H^1(S,\mathcal{O}_{S})$ denotes the irregularity of $S$ as usual.
When $q_{f}$ is positive, we call $f$ an {\it irregular} fibration.
Xiao showed in \cite{Xiao2} that $\lambda_f\geq 4$ holds for irregular fibrations.
It seems a general rule that the lower bound of the slope goes up, when the relative irregular gets bigger.

In the present papaer, we consider primitive cyclic covering fibrations of type $(g,h,n)$ introduced in \cite{Eno}, 
where Enokizono gave the lower bound of the slope for them.
Note that it is nothing more than a hyperelliptic fibration when $h=0$ and $n=2$.
Recall that Lu and Zuo obtained the lower bound of the slope for irregular double covering fibrations in \cite{X.Lu2} and \cite{X.Lu1}.
Inspired by their results, we try to generalize them to irregular primitive cyclic covering fibrations with $n\geq 3$.
We give the lower bound of slope for those of type $(g,0,n)$ 
in Theorems~\ref{prop2.3} and \ref{thm4.7}, 
and for those of type $(g,h,n)$ with $h\geq 1$ in Theorem $\ref{thm2.7}$.

The key observation for the proof is Proposition~\ref{prop2.2}.
We apply it to the anti-invariant part of the Albanese map 
with respect to the action of the Galois group canonically associated to the cyclic covering fibration, and  
derive the ``negativity'' of the ramification divisor when $q_f>0$.
Recall that $\chi_f$ and (the essential part of) $K_f^2$ can be expressed in terms of 
the so-called $k$-th {\it singularity index} $\alpha_k$ defined for each non-negative integer $k$.
The negativity referred above can be used to get some non-trivial restrictions on $\alpha_0$ which is the most difficult one to handle with 
among all $\alpha_k$'s. 
Thanks to such information together with an analysis of the Albanese map, we can obtain the desired slope inequalities.

We also give a small contribution to the modified Xiao's conjecture that $q_{f} \leq \lceil \frac{g+1}{2} \rceil$ holds, 
posed by Barja, Gonz\'{a}lez-Alonso and Naranjo in \cite{Bar1}.
It is known to be true for fibrations of maximal Clifford index \cite{Bar1} and for hyperelliptic fibrations \cite{X.Lu2} among others.
We show in Theorem \ref{xiaoconj} that $q_f\leq (g+1-n)/2$ holds, when $f$ is a primitive cyclic covering fibration of type $(g,0,n)$ 
under some additional assumptions. 
For the history around the conjecture, see the introduction of \cite{Bar1}.


The author express his sincere gratitude to Professor Kazuhiro Konno for suggesting 
this assignment, his valuable advice and support.  
The author also thanks Dr. Makoto Enokizono for his precious advices, allowing him to use 
Proposition~\ref{prop2.2} freely.
\section{Primitive cyclic covering fibrations}

We recall the basis properties of primitive cyclic covering fibrations, most of which can be found in \cite{Eno}.

\begin{define}
Let $f:S \to B$ be a relatively minimal fibration of genus $g\geq2$.
We call it a primitive cyclic covering fibration of type $(g,h,n)$, when 
there are  $($not necessarily relatively minimal$)$ fibration $\tilde{\PHI}:\WT \to B$ of 
genus $h\geq0$ and a classical $n$-cyclic covering
$$
\tilde{\theta}:\ST=
\mathrm{Spec}_{\WT}\left(\bigoplus_{j=0}^{n-1} \mathcal{O}_{\WT}(-j\DT)\right)\to\WT
$$
branched over a smooth curve $\RT \in |n\DT|$ for some 
$n\geq2$ and $\DT \in \Pic(\WT)$ such that $f$ is the relatively minimal model of 
$\tilde{f}=\tilde{\PHI}\circ\tilde{\theta}$.
\end{define}

Let $f:S \to B$ be a  primitive cyclic covering fibration of type $(g,h,n)$. 
Let $\widetilde{F}$ and $\widetilde{\Gamma}$ be general fibers of 
$\tilde{f}$ and $\tilde{\PHI}$, respectively.
Then the restriction map 
$\tilde{\theta}|_{\widetilde{F}}:\widetilde{F} \to \widetilde{\Gamma}$ is a classical 
$n$-cyclic covering branched over $\RT\cap\widetilde{\GA}$.
By the Hurwitz formula for $\tilde{\theta}|_{\widetilde{F}}$, we get 
\begin{equation}
r:=\RT.\widetilde{\GA}=\frac{2\bigl(g-1-n(h-1)\bigr)}{n-1}.
\end{equation}
From $\RT \in |n\DT|$, it follow that $r$ is a multiple of $n$.

Let $\tilde{\psi}: \WT \to W$ be the contraction morphism to a relative minimal 
model $W\to B$ of $\tilde{\PHI}:\WT \to B$.
Since $\tilde{\psi}$ is a composite of blowing-ups, we can write 
$\tilde{\psi}=\psi_{1}\circ\cdots\psi_{N}$, 
where $\psi_{i}:W_{i}\to W_{i-1}$ denotes the blowing-up 
at $x_{i} \in W_{i-1}\;(i=1,\cdots,N)$, $W_{0}=W$ and $W_{N}=\WT$.
We define a reduced curve $R_{i}$ inductively as $R_{i-1}=(\psi_{i})_{\ast}R_{i}$ 
starting from $R_{N}=\RT$ down to $R_{0}=R$.
We also put $E_{i}=\psi_{i}^{-1}(x_{i})$ and $m_{i}=\mathrm{mult}_{x_{i}}R_{i-1}\;(i=1,\cdots,N)$.

\begin{lem}[\cite{Eno}, Lemma 1.5]
\label{lem1.2}
In the above situation, the following hold for any $i=1,\cdots,N$.

\smallskip

$(1)$ Either $m_{i} \in n\Z$ or $n\Z+1$. 
Furthermore, $m_{i}\in n\Z$
if and only if $E_{i}$ is not contained in $R_{i}$.

\smallskip

$(2)$ $R_{i}={\psi}_{i}^{\ast}R_{i-1}-n[\frac{m_{i}}{n}]E_{i}$, where $[t]$ denotes the 
greatest integer not exceeding $t$.

\smallskip

$(3)$ There exists $\D_{i} \in \Pic(W_{i})$ such that $\D_{i}=\psi_{i}^{\ast}\D_{i-1}-[\frac{m_{i}}{n}]E_{i}$
 and $R_{i}\sim n\D_{i}$, $\D_{N}=\DT$. 

\end{lem}

\begin{rmk}
By \cite{Eno}, we can assume the following for any primitive cyclic covering 
fibrations.
Let $\SigmaT$ be a generator of the covering transformation group of $\tilde{\theta}$, and 
$\sigma$ the automorphism of $S$ over $B$ induced by $\SigmaT$.
Then the natural morphism $\rho:\ST\to S$ is a minimal succession of blowing-ups that resolves all isolated fixed points of $\sigma$.
\end{rmk}
 

We must pay a special attention when $h=0$, since we have various relatively minimal models for $\tilde{\PHI}:\WT \to B$.
Using elementary transformations, one can show the following.

\begin{lem}[\cite{Eno}, Lemma 3.1]
\label{lem1.4}
Let $f:S\to B$ be a primitive cyclic covering fibration of type $(g,0,n)$.
Then there is a relatively minimal model of $\tilde{\PHI}:\WT \to B$ such that 
$$
\mathrm{mult}_{x}R_{h}\leq\frac{r}{2}=\frac{g}{n-1}+1
$$
for all $x \in R_{h}$, where $R_{h}$ denotes the $\PHI$-horizontal part of $R$.
Moreover if $\mathrm{mult}_{x}R>\frac{r}{2}$, then 
$\mathrm{mult}_{x}R \in n\Z + 1$.
\end{lem}

When $h=0$, we always assume that a relatively minimal model of $\tilde{\PHI}:\WT \to B$ is as in the above lemma.

\begin{cor}
Let the situation be the same as in Lemma~\ref{lem1.4}. 
If $x$ is a singular point of $R$ and $m=\mathrm{mult}_{x}R$, then  
$$
n\biggl[\frac{m}{n}\biggr]\leq\frac{r}{2}.
$$
\label{lem1.4'}
\end{cor}

\begin{proof}
When $m \in n\Z$, the inequality clearly holds by Lemma $\ref{lem1.4}$.
If $m \in n\Z+1$, then $n[\frac{m}{n}]+1=m$.
From Lemma $\ref{lem1.4}$, we have $m \leq \frac{r}{2}+1$. 
So we get $n[\frac{m}{n}]\leq\frac{r}{2}$.
\end{proof}



In closing the section, we give an easy lemma that will be usuful in the sequel.

\begin{lem}
\label{lem3.2}
Let $\pi:C_{1}\to C_{2}$ be a surjective morphism between smooth projective curves. 
Let $R_{\pi}$ and $\Delta$ be the 
ramification divisor and the branch locus of $\pi$, respectively.
Then,
$$
(\deg(\pi)-1)\sharp\Delta \geq \deg R_{\pi},
$$
where $\sharp \Delta$ denotes the cardinality of $\Delta$ as a set of points.
\end{lem}

\begin{proof}
We put $\Delta=\{Q_{1},\;\dots\;,Q_{\sharp \Delta}\}$.
For any $Q_{i} \in \Delta$, we put $\pi^{-1}(Q_i)=\{P_{1}^{i},\;\dots\;,P_{j_{i}}^{i}\}$.
Note that $\deg(\pi) = r(P_{1}^{i})+\;\cdots\;+r(P_{j_{i}}^{i})$ for any $i=1,\;\dots\;,\sharp \Delta$,
where $r(P)$ denotes the ramification index of $\pi$ around $P\in C_{1}$.
Then, from the property of ramification divisor, 
\begin{align*}
\deg R_{\pi} &= \sum_{i=1}^{\sharp \Delta}\sum_{j=1}^{j_i}(r(P_j^{i})-1)\\
&= \deg(\pi)\sharp \Delta-(j_{1}+\;\cdots\;j_{\sharp \Delta})\\
& \leq (\deg(\pi)-1)\sharp \Delta,
\end{align*}
which is what we want.
\end{proof}

\section{Singularity indices and the formulae for $K_{f}^2$ and $\chi_{f}$.}

We let $f:S\to B$ be a primitive cyclic covering fibration of type $(g,h,n)$ and freely use the notation in the previous section. 
We obtain a classical $n$-cyclic covering
$\theta_{i}:S_{i} \to W_{i}$ branched over $R_{i}$ by setting 
$$S_{i} = \mathrm{Spec}\;(\bigoplus_{j=0}^{n-1}\mathcal{O}_{W_{i}}(-j \D_{i}))$$
Since $R_{i}$ is reduced, $S_{i}$ is a normal surface.
There exists a natural birational morphism $S_{i}\to S_{i-1}$. Set $S'=S_{0}$, $\theta=\theta_{0}$, $\D=\D_{0}$ and $f'=\PHI \circ \theta $.
Then we have a commutative diagram:
$$
\xymatrix{
\ST=S_N \ar[d]^{\widetilde{\theta}} \ar[r] \ar@(ur,ul)[rrrr]^{\rho} & S_{N-1} \ar[d]^{\theta_{N-1}} \ar[r] & \cdots \ar[r] & S_0=S^{\prime} \ar[d]^{\theta} & S \ar[ddl]^f\\
\WT=W_N \ar[r]^{\psi_N} \ar[drrr]^{\widetilde{\varphi}}                & W_{N-1} \ar[r]^{\psi_{N-1}}            & \cdots \ar[r]^{\psi_1} & W_0=W \ar[d]^{\varphi} \\
                                                                                                    &                                                &                              & B}$$
The well-known formulae for cyclic coverings give us 
\begin{align}
& K_{\tilde{f}}^2=n(K_{\tilde{\PHI}}^2+2(n-1)\tilde{\D}.K_{\tilde{\PHI}}+(n-1)^2 \tilde{\D}^2),
\label{(1.1)}\\
& \chi_{\tilde{f}}=n\chi_{\tilde{\PHI}}+\frac{1}{2}\sum_{j=0}^{n-1}j\tilde{\D}(j\tilde{\D}+K_{\tilde{\PHI}}),
\label{(1.2)}
\end{align}
(see e.g., \cite{Eno}).
From Lemma $\ref{lem1.2}$ and a simple calculation, we get
\begin{align}
\tilde{\D}^2=\D^2-\sum_{i=1}^{N}\biggl[ \frac{m_{i}}{n} \biggr]^2,
\label{(1.3)} \\
\tilde{\D}.K_{\tilde{\PHI}}=\D.K_{\PHI}+\sum_{i=1}^{N}\biggl[ \frac{m_{i}}{n} \biggr]
\label{(1.4)}
\end{align}
and
\begin{equation}
K_{\tilde{\PHI}}^2=K_{\PHI}^2-N.
\label{(1.5)}
\end{equation}

\begin{define}
\label{def1.1}
Let $\Gamma_p$ and $F_p$ respectively denote fibers of $\PHI:W\to B$ and $f:S\to B$ over a point $p\in B$.
For any fixed $p\in B$, we consider all singular points $($including infinitely near ones$)$ of $R$ on $\Gamma_{p}$. 
For any positive integer $k$, we let $\alpha_{k}(F_{p})$ be the number of singular points of multiplicity either $kn$ or $kn+1$ among them, 
and call it the $k$-th {\it singularity index} of $F_{p}$. 
We put $\alpha_{k}:=\sum_{p\in B}\alpha_{k}(F_{p})$ and call it the $k$-th singularity index of the fibration.
We also put $\alpha_{0}:=(K_{\tilde{\PHI}}+\RT)\RT$ and call it the ramification index of $\tilde{\PHI}|_{\RT}:\RT \to B$.
\end{define}

By a simple calculation, we get

\begin{align}
& N=\sum_{k\geq1}\alpha_{k},
\label{(1.6)} \\
& \sum_{i=1}^{N} \biggl[ \frac{m_{i}}{n} \biggr]=\sum_{k\geq1}k \alpha_{k},
\label{(1.7)} \\
& \sum_{i=1}^{N}  \biggl[ \frac{m_{i}}{n} \biggr]^2=\sum_{k\geq1}k^2 \alpha_{k}
\label{(1.8)}
\end{align}
and
\begin{equation}
\alpha_{0}=(K_{\PHI}+R)R-\sum_{k\geq1}nk(nk-1)\alpha_{k}.
\label{(1.9)}
\end{equation}

Substituting (\ref{(1.3)}) through (\ref{(1.8)}) for (\ref{(1.1)}) and (\ref{(1.2)}), one gets
$$
K_{\tilde{f}}^2=n(K_{\PHI}^2+2(n-1)\D.K_{\PHI}+(n-1)^2 \D^2)-n\sum_{k\geq1}((n-1)k-1)^2\alpha_{k}
$$
and
$$
\chi_{\tilde{f}}=n\chi_{\tilde{\PHI}}+\frac{n(n-1)}{4}\D.K_{\PHI}+\frac{n(n-1)(2n-1)}{12} \D^2-\frac{n(n-1)}{12}\sum_{k\geq1}((2n-1)k^2-3k)\alpha_{k}.
$$
Since $K_{f}^2\geq K_{\tilde{f}}^2$ , $\chi_{\tilde{f}}=\chi_{f}$ and $\chi_{\tilde{\PHI}}=\chi_{\PHI}$, we obtain
\begin{equation}
K_{f}^2\geq n(K_{\PHI}^2+2(n-1)\D.K_{\PHI}+(n-1)^2 \D^2)-n\sum_{k\geq1}((n-1)k-1)^2\alpha_{k}
\label{(1.10)}
\end{equation}
and
\begin{equation}
\chi_{f}=n\chi_{\PHI}+\frac{n(n-1)}{4}\D.K_{\PHI}+\frac{n(n-1)(2n-1)}{12} \D^2\\
-\frac{n(n-1)}{12}\sum_{k\geq1}((2n-1)k^2-3k)\alpha_{k}.
\label{(1.11)}
\end{equation}

We treat the cases $h=0$ and $h>0$ separately.

\begin{prop}
\label{prop1.2}
Let $f:S\to B$ be a primitive cyclic cover fibration of type $(g,0,n)$ and let $\alpha_{i}$ $(i\geq 0 )$ be the singularity index in Definition $\ref{def1.1}$. Then,
 \begin{align}
(r-1)K_{f}^2 & \geq  \frac{(r-2)n-r}{n} (n-1)\alpha_{0}\label{(1.12)} \\
& + \sum_{k\geq 1}^{nk\leq \frac{r}{2}}  \biggl( \frac{n^2-1}{n} nk(r-1-(nk-1))-(r-1)n \biggr) \alpha_{k},
\nonumber
\end{align}
\begin{align}
(r-1)\chi_{f}^2= \frac{(2r-3)n-r}{12n} (n-1)\alpha_{0} + \sum_{k\geq 1}^{nk\leq \frac{r}{2}}  \biggl( \frac{n^2-1}{12n}nk(r-1-(nk-1)) \biggr) \alpha_{k}.
\label{(1.13)}
\end{align}
\end{prop}

\begin{proof}
Note that if $\AL_{k}>0$, then $nk\leq \frac{r}{2}$ from Corollary \ref{lem1.4'}. 

We find that $R \equiv -\frac{r}{2}K_{\PHI}+M_0\Gamma$ for some $M_0 \in \frac{1}{2} \Z$, 
where the symbol $\equiv$ means the numerical equivalence,
since $\PHI:W \to B$ is a $\mathbb{P}^1$-bundle and we have $K_{W}.\Gamma = -2$ and  $\RT.\widetilde{\GA}=R.\GA = r$. Hence we get 
\begin{align}
R.K_{\PHI}=-2M_0,
\label{(1.14)}
\end{align}
\begin{align}
R^2=2rM_0.
\label{(1.15)}
\end{align}
Therefore we have 
$$
(K_{\PHI}+R).R=2(r-1)M_0.
$$
From this equality and (\ref{(1.9)}), we get 
\begin{align}
2(r-1)M_0=\alpha_{0}+\sum_{k\geq 1}^{nk\leq \frac{r}{2}}nk(nk-1)\alpha_{k}.
\label{(1.16)}
\end{align}
On the other hand, substituting (\ref{(1.14)}) and (\ref{(1.15)}) for (\ref{(1.10)}), we get 
$$
K_{f}^2\geq \frac{(r-2)n-r}{n}2(n-1)M_0 -n \sum_{k\geq 1}^{nk\leq \frac{r}{2}} ((n-1)k-1)^2\alpha_{k}.
$$
Multiplying this by $r-1$ and substituting $(\ref{(1.16)})$ for it, we get (\ref{(1.12)}).
Similarly one can show (\ref{(1.13)}).
\end{proof}

When $h>0$, we have the following:

\begin{prop}
\label{prop1.3}
Let $f:S\to B$ be a primitive cyclic covering fibration of type $(g,h,n)$ such that $h\geq1$ and 
$\alpha_{i}$ $(i\geq 0 )$ the singularity index in Definition $\ref{def1.1}$. 
 Put
$$
t:=2(g-1)-(h-1)(n+1),
$$
$$
T:=\left\{
    \begin{array}{lcr}
     2(g-1)K_{\PHI}.R,  & \text{if }h=1, \\
     -\frac{\bigl((g-1-n(h-1))K_{\PHI}-(n-1)(h-1)R\bigr)^2}{(n-1)(h-1)}, &\text{if }h\geq2.
    \end{array}
  \right.
$$
Then $t>0$ and $T\geq 0$.
Furthermore, 
\begin{align}
\numberwithin{equation}{section}
tK_{f}^2\geq tx'\frac{K_{\PHI}^2}{(n-1)(h-1)}+ty'T+tz'\AL_{0}+\sum_{k\geq1}a_{k}\AL_{k}
\label{1.17}
\end{align}
and
\begin{align}
\numberwithin{equation}{section}
t\chi_{f}=nt\chi_{\PHI}+ t\bar{x}'\frac{K_{\PHI}^2}{(n-1)(h-1)}+t\bar{y}'T+t\bar{z}'\AL_{0}+\sum_{k\geq1}\bar{a}_{k}\AL_{k}
\label{1.18}
\end{align}
hold, where
\begin{align*}
x'&=\frac{(g-1)(n-1)\bigl((g-1)(n+1)-2n(h-1)\bigr)}{nt},
\quad\bar{x}'=\frac{(n-1)(n+1)(g-1-n(h-1))^2}{12nt},\\
y'&=\frac{(n-1)\bigl(2-\frac{n-1}{n}\bigr)}{t},
\quad\bar{y}'=\frac{(n-1)(n+1)}{12nt},\\
z'&=\frac{2(n-1)^2(g-1)}{nt},
\quad\bar{z}'=\frac{(n-1)\bigl(2(2n-1)(g-1)-n(n+1)(h-1)\bigr)}{12nt},
\end{align*}
\begin{align*}
a_{k}=12\bar{a_{k}}-nt,\quad
\bar{a}_{k}=\frac{k}{12}(n^{2}-1)\bigl(2(g-1)+n(h-1)((n-1)k-2)\bigr).
\end{align*}
In $(\ref{1.17})$ and $(\ref{1.18})$, the quantity $\frac{K_{\PHI}^2}{(n-1)(h-1)}$ is understand to be zero, when $h=1$.
\end{prop}

\begin{proof}
We get $t>0$ from $r\geq 0$, $n\geq 2$ and $g\geq 2$.
We shall show that $T\geq0$.
If $h=1$, by the canonical bundle formula, we have
$$K_{\PHI}\equiv\chi(\mathcal{O}_{W})\GA + \sum_{i=1}^{l}\bigl(1-\frac{1}{k_{i}}\bigr)\GA$$
where $\{ k_{i}\mid i=1, \dots,l \}$ denotes the set of multiplicities of all multiple fibers of $\PHI$, $k_{i}\geq 2$.
Hence we get
\begin{align}
\numberwithin{equation}{section}
K_{\PHI}.R \geq \chi(\mathcal{O}_{W})\GA.R=\frac{2(g-1)}{n-1}\chi(\mathcal{O}_{W})
\label{1.19}
\end{align}
Since $W$ is an elliptic surface, we have $\chi(\mathcal{O}_{W})\geq0$ and, hence, $T\geq0$.
If $h\geq2$, we consider the intersection matrix 
\[
  \left(
    \begin{array}{ccc}
      K_{\PHI}^2 & K_{\PHI}.\D & K_{\PHI}.\GA \\
       K_{\PHI}.\D & \D^2 & \D.\GA \\
       K_{\PHI}.\GA & \D.\GA & 0
    \end{array}
  \right)
\]
for $\{K_{\PHI},\;\D,\;\GA \}$.
Since we have $K_{\PHI}^2\geq0$ by Arakelov's theorem, it is not negative definite. 
Hence its determinant is non negative by the Hodge index theorem, and we get
\begin{align}
\numberwithin{equation}{section}
2(K_{\PHI}.\D)(\D.\GA)(K_{\PHI}.\GA)-\D^2(K_{\PHI}.\GA)^2-(\D\GA)^2K_{\PHI}^2\geq 0.
\label{1.20}
\end{align}
Since 
\begin{align*}
\D.\GA&=\frac{r}{n}=\frac{2(g-1-n(h-1))}{n(n-1)},\;\;
K_{\PHI}.\GA=2(h-1),
\end{align*}
the inequality $(\ref{1.20})$ is equivalent to 
$$
2\bigl(g-1-n(h-1)\bigr)K_{\PHI}.\D-n(n-1)(h-1)\D^2
\geq \frac{1}{n(n-1)(h-1)}\bigl(g-1-n(h-1)\bigr)^2K_{\PHI}^2.
$$
So we get 
$$
0\geq \bigl((g-1-n(h-1))K_{\PHI}-(n-1)(h-1)R\bigr)^2
$$
and, hence, $T \geq 0$.

Now, by a direct calculation, one has 
$$n\bigl(K_{\PHI}^2+\frac{2(n-1)}{n}K_{\PHI}.R+(\frac{n-1}{n})^2R^2\bigr)
= x'\frac{K_{\PHI}^2}{(n-1)(h-1)}+y'T+z'(K_{\PHI}+R)R$$
and
$$n\chi_{\PHI}+\frac{n-1}{4}K_{\PHI}.R+\frac{(n-1)(2n-1)}{12n}R^2=
n\chi_{\PHI}+\bar{x}'\frac{K_{\PHI}^2}{(n-1)(h-1)}+\bar{y}'T+\bar{z}'(K_{\PHI}+R)R.$$
Hence we obtain from $(\ref{(1.10)})$ and $(\ref{(1.11)})$ that
\begin{align}
\numberwithin{equation}{section}
K_{f}^2\geq x'\frac{K_{\PHI}^2}{(n-1)(h-1)}+y'T+z'(K_{\PHI}+R)R-n\sum_{k\geq1}\bigl((n-1)k-1\bigr)^2\AL_{k}
\label{1.21}
\end{align}
and
\begin{align}
\numberwithin{equation}{section}
\chi_{f}=n\chi_{\PHI}+\bar{x}'\frac{K_{\PHI}^2}{(n-1)(h-1)}+\bar{y}'T+\bar{z}'(K_{\PHI}+R)R-\frac{n(n-1)}{12}\sum_{k\geq1}\bigl((2n-1)k^2-3k\bigr)\AL_{k}.
\label{1.22}
\end{align}
From ($\ref{(1.9)}$) and $(\ref{1.21})$, we get 
$$
tK_{f}^2\geq tx'\frac{K_{\PHI}^2}{(n-1)(h-1)}+ty'T+tz'\AL_{0}+
\sum_{k\geq1}\bigl(2(n-1)^2(g-1)k(nk-1)-nt((n-1)k-1)^2\bigr)\AL_{k}
$$
Since one sees 
$a_{k}=2(n-1)^2(g-1)k(nk-1)-nt((n-1)k-1)^2$, we obtain $(\ref{1.17})$.
Similarly, we obtain $(\ref{1.18}).$
\end{proof}

\section{Slope inequality for irregular cyclic covering fibrations.}

The purpose of this section is to show the slope inequalities for irregular cyclic 
covering fibrations of type $(g,h,n)$, $n\geq 3$.
We start things in a more general setting.

\begin{define}
\label{def2.1}
Let $\tilde{\theta}:\tilde{S} \to \tilde{W} $ be a finite Galois cover $($not necessarily primitive cyclic$)$ between smooth projective varieties with Galois group  $G$.
Let $\alpha:\ST \to \Alb(\ST)$  be the Albanese map. 
For any $\SigmaT \in G$,
we denote by $\alpha({\tilde{\sigma}}):\Alb(\ST)\to \Alb(\ST)$ the morphism 
induced from $\SigmaT:\ST\to \ST$ by the universality of the Albanese map.
We put 
$$
\Alb_{\SigmaT}(\ST):=\Image \{\alpha(\SigmaT)-1:\Alb(\ST)\to \Alb(\ST)\}
$$
and let 
$$
\alpha_{\SigmaT}:\ST\to \Alb_{\SigmaT}(\ST)
$$
be the morphism defined by $\alpha_{\SigmaT}:=(\alpha(\SigmaT)-1) \circ \alpha$.
\end{define}

The following is due to Makoto Enokizono.

\begin{prop}
\label{prop2.2}
Suppose that $G$ is a cyclic group generated by $\SigmaT$ in the above situation. 
If $q_{\widetilde{\theta}}:=q(\widetilde{S})-q(\widetilde{W})>0$, then the following hold.

\smallskip

$(1)$ $\dim \Alb_{\SigmaT}(\ST)=q_{\widetilde{\theta}}$. 

\smallskip
 
$(2)$ If $\Fix(G):=\{x\in\ST \mid \SigmaT(x)=x \}\neq \emptyset$, then it is contracted  by $\alpha_{\SigmaT}$ to the unit element $0 \in \Alb_{\SigmaT}(\ST)$.

\smallskip

$(3)$ If $\alpha_{\SigmaT}(\ST)$ is a curve, then the geometric genus of 
$\alpha_{\SigmaT}(\ST)$ is not less than $q_{\tilde{\theta}}$.
\end{prop}

\begin{proof}
Firstly, we show (1). 
By the construction of $\alpha(\SigmaT)-1:\Alb(\ST)\to \Alb(\ST)$, we get the following commutative diagram:
\begin{equation}
\nonumber
 \vcenter{ \xymatrix{ 
 H^0(\Alb(\ST),\Omega_{\Alb(\ST)}^1) \ar[r]^{(\AL(\SigmaT)-1)^{\ast}} \ar[d]_{\AL^{\ast}} 
 & H^0(\Alb(\ST),\Omega_{\Alb(\ST)}^1) \ar[d]^{\AL^{\ast}} \\
 H^0(\ST,\Omega_{\ST}^1) \ar[r]_{\SigmaT^{\ast}-1} 
 & H^0(\ST,\Omega_{\ST}^1).} } 
 \end{equation}
Since $\alpha^*$ is an isomorphism, we have $\dim \Ker(\alpha(\SigmaT)-1)^{\ast} = \dim \Ker(\SigmaT^{\ast}-1)$. 
Since $G$ is a cyclic group generated by $\SigmaT$, we see that 
$\Ker(\SigmaT^{\ast}-1)$ coincides with 
the $G$-invariant part $H^0(\ST,\Omega^1_{\ST})^{G}$ of ${H}^0(\ST,\Omega^1_{\ST})$.
On the other hand, since  
$\tilde{\theta}^{\ast}:\mathrm{H}^0(\WT,\Omega^1_{\WT})\to \mathrm{H}^0(\ST,\Omega^1_{\ST})^{G}$
is an isomorphism, we have $\dim(\Ker(\alpha(\SigmaT)-1)^{\ast} )= q(\WT)$ and, hence,
$$
\dim(\Image(\alpha(\SigmaT)-1)^{\ast}) = q(\ST)-q(\WT).
$$ 
It follows that $\Alb_{\SigmaT}(\ST)$ is of dimension $q_{\widetilde{\theta}}$.

Secondly, we show (2).
We take a point $x_{0}$ in $\Fix(G)$ as the base point of the Albanese map $\alpha:\ST \to \Alb(\ST).$ 
Let $x \in \Fix(G)$. 
Note that we have
$$
\alpha_{\SigmaT}(x)=(\alpha(\SigmaT)-1)(\alpha(x))=\alpha(\SigmaT)(\alpha(x))-\alpha(x)
$$
and that $\alpha(\SigmaT)(\alpha(x))-\alpha(x)$ is the function 
given by $\omega \mapsto \int_{x_{0}}^{x}\SigmaT^ {\ast}\omega - \int_{x_{0}}^{x}\omega $ for $\omega \in \mathrm{H}^0(\ST,\Omega^1_{\ST})$ modulo 
periods. 
Since $x$ and $x_{0} $ are both in $\Fix(G)$, we find 
$$
\int_{x_{0}}^{x}\SigmaT^ {\ast}\omega - \int_{x_{0}}^{x}\omega=\int_{\SigmaT(x_{0})}^{\SigmaT(x)}\omega - \int_{x_{0}}^{x}\omega=0.
$$
Hence we get $\alpha_{\SigmaT}(x)=\alpha(\SigmaT)(\alpha(x))-\alpha(x)=0$.
Since $x$ can be taken arbitrarily in $\Fix(G)$, we get (2).

When $\alpha_{\SigmaT}(\ST)$ is a curve, one can check easily that the geometric genus of $\alpha_{\SigmaT}(\ST)$ is not less than $q_{\tilde{\theta}}$,  
by (1) and the universality of the Abel-Jacobi map for the normalization of $\alpha_{\SigmaT}(\ST)$.
Hence (3).
 \end{proof}

\subsection{The slope inequality in the case of $h=0$}

We consider the primitive cyclic covering fibration $f:S \to B$ of type $(g,0,n)$ with $q_{f}>0$. 
Since $\PHI:W \to B$ is a ruled surface, we have $q(W)=b$ and it follows $q_{\tilde{\theta}}=q_{f}$.
We apply Proposition~\ref{prop2.2} to the cyclic covering $\tilde{\theta}:\ST \to \WT$ to find that 
its ramification divisor $\Fix (G)$ is contracted to a point by $\AL_{\SigmaT}: \ST\to \Alb_{\SigmaT}(\ST)$, 
where $\SigmaT$ denotes a generator of the Galois group $G:=\mathrm{Gal}(\ST/\WT)$.
So if $\alpha_{\SigmaT}(\ST)$ is a surface (resp. a curve), from Mumford's theorem (resp. Zariski's Lemma), 
the intersection form is negative definite (resp. semi-definite) on $\Fix (G)$, and we in particular get
$$
\Fix (G)^2 < 0\quad(\textrm{resp.}\leq0).
$$
Hence, in any way, we have  
\begin{align}
\RT^2 \leq 0,
\label{(2.1)}
\end{align}
since $\tilde{\theta}^*\RT=n\Fix (G)$.

Here, we remark the following.

\begin{lem}\label{lem_added}
Let $f:S\to B$ be a primitive cyclic covering of type $(g,0,n)$. 
If it is not locally trivial and $q_f>0$, then $r\geq 2n$.
\end{lem}

\begin{proof}
We assume that $r<2n$ and show that this leads us to a contradiction.

Recall that $r$ is a multiple of $n$.
If $r<2n$, then $r\leq n$ and $R$ has to be smooth by Lemmas~\ref{lem1.2} and \ref{lem1.4}.
On the other hand, since $q_f>0$, we already know from (\ref{(2.1)}) that the self-intersection number of any irreducible component $C$ of $R$ is 
non-positive.

Let $C_0$ be the minimal section with $C_0^2=-e$ and $\Gamma$ a fiber of $\varphi:W\to B$.
Note that we can choose the normalized vector bundle of rank $2$ associated with $W$ so that there are no effective divisor numerically equivalent to 
$C_0-c\Gamma$ for any positive integer $c$.
Put $C\equiv aC_0+b\Gamma$ with two integers $a$, $b$. We have $a\geq 0$.
If $a=0$, then we have $b=1$, that is, $C$ is a single fiber by its irreducibility. 
So we may assume that $a$ is positive.

We have $C^2=a(2b-ae)\leq 0$.
Hence $2b\leq ae$.
Furthermore, since $C$ is irreducible and $a>0$, we have $(K_\varphi+C)C\geq 0$ by the Hurwitz formula applied for 
the normalization of $C$. Since $K_\varphi\equiv -2C_0-e\Gamma$, we have
$$
0 \leq (K_\varphi+C)C =(a-1)(2b-ae)\leq 0,
$$
by $2b\leq ae$ and $a\geq 1$.
Hence we get $(K_\varphi+C)C=0$ and, either $a=1$ or $2b=ae$.
In particular, as the first equality shows, $C$ is smooth and $\varphi|_C:C\to B$ is unramified.
Furthermore, we get $C^2=0$ when $a\geq 2$ by $2b=ae$.
In this case, we also have $b\leq 0$, because $0\leq CC_0=b-ae=-b$.
If $a=1$ and $2b<e$, then $b\geq 0$ and it follows from $CC_0=b-e<-b\leq 0$ that we have $C=C_0$ by the irreducibility of $C$. 
We remark here that we have $(a_1C_0+b_1\Gamma)(a_2C_0+b_2\Gamma)=0$ when $a_i>0$ and $2b_i=a_ie$ for $i=1,2$.

In summary, the only possibilities left for smooth $R$ are (i) $R$ consists of several fibers (including the case $R=0$), 
(ii) $R$ is the minimal section with $R^2<0$, and 
(iii) $R$ consists of several smooth curves with self-intersection numbers $0$ which are unramified over $B$ (via $\varphi$).
If (i) or (ii) is the case, then we have either $g=0$ or $r=1$, any of which is absurd.
If (iii) is the case, then $f:S\to B$ is a locally trivial fibration, which is again inadequate.
\end{proof}

From $\RT^2 = R^2 - \sum_{k \geq 1} n^2k^2 \alpha_{k}$, (\ref{(1.15)}) and (\ref{(2.1)}), we get
$$
2rM_0 \leq  \sum_{k \geq 1}^{nk\leq \frac{r}{2}} n^2k^2 \alpha_{k}.
$$
Hence from this and (\ref{(1.16)}), we get 
\begin{align}
\alpha_{0} \leq \sum_{k \geq 1}^{nk\leq \frac{r}{2}} \frac{nk}{r}(r-1-(nk-1))\alpha_{k}.
\label{(2.2)}
\end{align}

\begin{thm}
\label{prop2.3}
Let $f:S \to B$ be a primitive cyclic covering fibration of type $(g,0,n)$ which is not locally trivial and $q_{f}>0$. 
Then 
$$
\lambda_{f} \geq \lambda_{g,n}^1:= 8- \frac{8(g+n-1)}{(n-1)(2g-(n-1)(n-2))}\biggl(=8-\frac{4r}{(n-1)(r-n)}\biggr).
$$
\end{thm}

\begin{proof}
For $\lambda \in \mathbb{R}$, we put 
$$
A(\lambda):= \frac{n-1}{n}\biggl((r-2)n-r-\lambda\frac{(2r-3)n-r}{12}\biggr).
$$
From Proposition $\ref{prop1.2}$, we get 
$$
(r-1)(K_{f}^2-\lambda_{g,n}^1\chi_{f})\geq A(\lambda_{g,n}^1)\alpha_{0} 
+ \sum_{k\geq1}^{nk\leq \frac{r}{2}}a_{k}\alpha_{k}-\sum_{k\geq1}^{nk\leq \frac{r}{2}}\lambda_{g,n}^1\bar{a}_{k}\alpha_{k}
$$
where 
$$
a_{k}:=(n^2-1)k(r-1-(nk-1))-(r-1)n,
$$
$$
\bar{a}_{k}:=\frac{n^2-1}{12}k(r-1-(nk-1)).
$$
We can check that $A(\lambda_{g,n}^1)\leq 0$ as follows. 
Since $r\geq 2n$ by Lemma~\ref{lem_added}, a calculation shows that the inequality 
$$
A(\lambda_{g,n}^1)= \frac{n-1}{n}\biggl((r-2)n-r-\bigl(8-\frac{4r}{(n-1)(r-n)}\bigr)\frac{(2r-3)n-r}{12}\biggr) \leq0
$$
is equivalent to 
$$
(r-n)(-(n-1)^2+2)+2n^2-3n-r\leq 0,
$$ 
and we can check easily its validity.
Therefore $A(\lambda_{g,n}^1) \leq 0$.
Hence from $(\ref{(2.2)})$, we get 
\begin{align}
(r-1)(K_{f}^2-\lambda_{g,n}^1\chi_{f})\geq \sum_{k \geq1}^{nk\leq\frac{r}{2}}\biggl(\frac{(n-1)(r-1)}{4r}(8-\lambda_{g,n}^1)nk(r-nk)-(r-1)n\biggr)\alpha_{k}.
\label{2.2'}
\end{align}
For any integer $k$ satisfying $\frac{r}{2n} \geq k \geq1$, we have $nk(r-nk)\geq n(r-n)$.
Since we have
\begin{align*}
\frac{(n-1)(r-1)}{4r}(8-\lambda_{g,n}^1)n(r-n)-(r-1)n=0,
\end{align*}
the coefficient of $\AL_{k}$ in the right hand side of $(\ref{2.2'})$ is not negative.
Therefore, we get $K_{f}^2-\lambda_{g,n}^1\chi_{f}\geq 0$ as desired.
\end{proof}

\subsection{The slope inequality in the case of $h\geq 1$.}
Before showing the slope inequality when $h\geq1$ and $n\geq 3$,
we study the upper bound of $\AL_{0}$.

Recall that we decomposed $\psi$ into a succession of blowing-ups $\psi_{i}$ as,
$$
\psi:\WT=W_{N} \overset{\psi_{N}}{\to} W_{N-1} \to \cdots \to W_1
\overset{\psi_{0}}{\to}W_{0} = W
$$
We define the order of blowing-up $\psi'$ appearing in $\psi$ as follows.
If the center of $\psi'$ is a point on the branch locus of multiplicity $m'$, we put  
$$
\Ord(\psi'):=\biggl[ \frac{m'}{n} \biggr].
$$
Moreover we introduce a partial order on these blowing-ups $\psi'$ and $\psi''$ appearing in $\psi$,
$$
\psi' \geq \psi''
\overset{\mathrm{def}}{\Longleftrightarrow}
\Ord(\psi')\geq \Ord(\psi'').
$$

\begin{lem}
\label{lem2.4}
Assume that $n\geq 3$.
Let $x_{j}$ $( \in R_{j} \subset W_{j})$ be a singular point infinitely near to $x_{i} \in R_{i}$.
Then the multiplicities satisfy $m_{j} \leq m_{i}$.  
\end{lem}

\begin{proof}
Though this can be found in \cite{Eno}, Lemma~3.7, when $n=0$, we shall give a proof 
for the convenience of readers.

Let $x_{i+1}$ be the singular point of $R_{i+1}$ infinitely near to $x_{i} \in R_{i}$.
If $m_{i} \in n\Z$, then $R_{i+1}$ coincied with $\widehat{R_{i}}$, 
the proper transform of $R_{i}$ by $\psi_{i+1}$, by Lemma $\ref{lem1.2}$.
Hence $m_{i+1}\leq m_{i}$ in this case.
If $m_{i} \in n\Z$+1, then $R_{i+1}=\widehat{R_{i}}+E_{i+1}$.
Hence we get $m_{i+1}\leq m_{i}+1 \in n\Z+2$. 
From Lemma $\ref{lem1.2}$ and the assumption $n\geq 3$, we get $m_{i+1}\leq m_{i}$.
\end{proof}

From Lemma~\ref{lem2.4}, we can reorder those blowing-ups appearing in $\psi$ 
so that
$\psi_{i}\geq\psi_{j}$ holds whenever $i<j$.
We put, 
$$
M:=\mathrm{max}\{\Ord(\psi') \mid \text{$\psi'$ is a blowing up in $\psi$} \}.
$$
Then we can decompose $\psi$ as
$$
\psi:\WT=\widehat{W}_{M} \overset{\hat{\psi}_{M}}{\to} \widehat{W}_{M-1} 
\quad \cdots \quad\overset{\hat{\psi}_{0}}{\to} \WH_{0} = W
$$
in such a way that $\Ord(\psi')=M+1-i$ holds for any $\psi'$ appearing in $\hat{\psi}_{i}$.

\begin{lem}
\label{lem2.5}
Let $\psi'$ be a blowing-up appearing in $\hat{\psi}_{i}$ and $\widetilde{D}$   
the proper inverse image of the exceptional curve of $\psi'$ on $\ST$.
Then the geometric genus of $\widetilde{D}$ satisfies
$$
g(\widetilde{D})\leq \frac{(n-1)(n(M-i)+n-2)}{2}.
$$
\end{lem}

\begin{proof}
Let $m'$ be the multiplicity of the singular point blown up by $\psi'$, and $\widetilde{E}$ the proper transform of the 
exceptional curve of $\psi'$ on $\WT$.

When $m'\in n\Z+1$,
since $\widetilde{E}$ is contained in $\RT$, $\widetilde{D}$ is a smooth rational curve.

Assume that $m'\in n\Z$.
From $m'=n(M+1-i)$, the intersection number of the exceptional curve of $\psi'$ 
and the branch locus is $n(M+1-i)$. 
Hence the intersection number of their proper transforms on $\WT$ is at most $n(M+1-i)$.
On the other hand, we consider the composite 
$\pi:\widetilde{D'} \to \widetilde{D} \overset{\tilde{\theta}|_{\widetilde{D}}}{\to} \widetilde{E}$, 
where $\widetilde{D'} \to \widetilde{D}$ is normalization of $\widetilde{D}$, and let $B_{\pi}$ be 
the branch locus $\pi$.
From the Hurwitz formula for $\pi$ and Lemma $\ref{lem3.2}$, we get 
$$
2g(\widetilde{D'})-2+2n \leq (n-1)\sharp B_{\pi}.
$$
Since $\tilde{\theta}|_{\tilde{D}}$ is totally ramified, 
\begin{align*}
\sharp B_{\pi} \leq \sharp B_{\tilde{\theta}|_{\tilde{D}}} \leq \widetilde{E}.\RT \leq n(M+1-i).
\end{align*}
Therefore we get
$$
g(\widetilde{D})=g(\widetilde{D'}) \leq \frac{(n-1)(n(M-i)+n-2)}{2},
$$
which is what we want.
\end{proof}

\begin{prop}
\label{prop2.6}
Let $f:S \to B$ be a primitive cyclic covering fibration of type $(g,h,n)$ such that 
$q_{\tilde{\theta}}>0$, $h\geq1$, $n\geq 3$ and let $\alpha_{i}$ $(i\geq 0 )$ be the singularity index in Definition $\ref{def1.1}$. Then,
\begin{align}
\nonumber
 &2\bigl(g-1-n(h-1)\bigr)\AL_{0}\\
&\leq\bigl(g-1-n(h-1)\bigr)^2 \frac{K_{\PHI}^2}{(n-1)(h-1)}+T+
\sum_{k\geq1}\frac{12n}{(n-1)(n+1)}\bar{a}_{k}\AL_{k},
\label{2.4}
\end{align}
where $\bar{a}_{k}$ is defined in Proposition $\ref{prop1.3}$.
If the image $\AL_{\SigmaT}(\ST)$ is a curve and $\nu(q_{\tilde{\theta}})\geq 1$, where
$$
\nu(x):=\biggl[ \frac{2(x-1)}{n(n-1)}-\frac{n-2}{n}\biggr],
$$
then
\begin{align}
\nonumber
 &2\bigl(g-1-n(h-1)\bigr)\bigl(\AL_{0}+\sum_{k=1}^{\nu(q_{\tilde{\theta}})}nk(nk-1)\AL_{k}\bigr) \\
&\leq\bigl(g-1-n(h-1)\bigr)^2 \frac{K_{\PHI}^2}{(n-1)(h-1)}+T+
\sum_{k\geq\nu(q_{\tilde{\theta}})+1}\frac{12n}{(n-1)(n+1)}\bar{a_{k}}\AL_{k}.
\label{2.5}
\end{align}
In $(\ref{2.4})$ and $(\ref{2.5})$, the quantity $\frac{K_{\PHI}^2}{(n-1)(h-1)}$ is understood to be zero when $h=1$.
\end{prop}

\begin{proof}
Firstly assume that $\AL_{\SigmaT}(\ST)$ is a curve of geometric genus $g'$.
In this case, by Proposition~\ref{prop2.2}, we have $g' \geq q_{\tilde{\theta}}$ and see that 
any curve of geometric genus less than $g'$ on $\widetilde{S}$ is contracted by
$\AL_{\SigmaT}$.
Hence, we know from Lemma $\ref{lem2.5}$ that for any $1\leq i \leq M$ satisfying 
$$
\frac{(n-1)(n(M-i)+n-2)}{2}\leq g'-1,
$$ 
the proper transform of the exceptional curve of $\hat{\psi}_{i}$ to $\ST$ is contracted 
by $\AL_{\SigmaT}$.
Then, since $q_{\tilde{\theta}}\leq g'$, for any $1\leq i\leq M$ satisfying 
$$
\frac{(n-1)(n(M-i)+n-2)}{2}\leq q_{\tilde{\theta}}-1,
$$ 
the same holds true.
So we conclude that the total inverse image of 
$\RH_{M-\nu(q_{\tilde{\theta}})}$ in $\ST$ is contracted by $\AL_{\SigmaT}$, where 
$\RH_{M-\nu(q_{\tilde{\theta}})}\subset \WH_{M-\nu(q_{\tilde{\theta}})}$ is the image of $\RT$. 
Therefore, the total inverse image of $\RH_{M-\nu(q_{\tilde{\theta}})}$ forms a negative 
semi-definite configuration.
In particular, we have 
\begin{align}
\numberwithin{equation}{section}
\RH_{M-\nu(q_{\tilde{\theta}})}^2 \leq 0.
\label{2.6}
\end{align}
By the construction, we have 
\begin{align*}
\RH_{M-\nu(q_{\tilde{\theta}})}^2&=
R^2-\sum_{k>\nu(q_{\tilde{\theta}})}n^2k^2\AL_{k}\\
&=\hat{x}\frac{K_{\PHI}^2}{(n-1)(h-1)}+\hat{y}T+\hat{z}(K_{\PHI}+R)R
-\sum_{k>\nu(q_{\tilde{\theta}})}n^2k^2\AL_{k},
\end{align*}
where 
$$
\hat{x}=-\frac{\bigl(g-1-n(h-1)\bigr)^2}{t},
\;\hat{y}=-\frac{1}{t},
\;\hat{z}=\frac{\bigl(g-1-n(h-1)\bigr)}{t}.
$$
Hence from $(\ref{(1.9)})$, $(\ref{2.6})$ and the above equality, 
we get  
$$
t\hat{z}\bigl(\AL_{0}+\sum_{k=1}^{\nu(q_{\tilde{\theta}})}nk(nk-1)\AL_{k} \bigr)
\leq -t\hat{x}\frac{K_{\PHI}^2}{(n-1)(h-1)}-t\hat{y}T+
\sum_{k > \nu(q_{\tilde{\theta}})}t\bigl( n^2k^2-nk(nk-1)\hat{z}\bigr)\AL_{k}.
$$
This is nothing more than $(\ref{2.5})$.

Even when $\AL_{\SigmaT}(\ST)$ is not a curve, we have 
$\RT^2\leq0$ by Proposition $\ref{prop2.2}$.
Using this instead of (\ref{2.6}), we get $(\ref{2.4})$ by a similar argument as above.
\end{proof}

\begin{thm}
\label{thm2.7}
Let $f:S\to B$ be a primitive cyclic covering fibration of type $(g,h,n)$ which is not locally trivial and such that 
$q_{\tilde{\theta}}$ and $h$ are both positive, $n\geq 3$. 
Put
$$
F(g,h,l)=(g-1)^2-2n(g-1)\bigl((h+1)(n-1)(l+1)-1 \bigr)-\bigl((l+1)(n-1)^2-1\bigr)^2 n^2(h^2-1).
$$

{\rm (i)} If  $F(g,h,0)\geq0$, then 
$$
\lambda_{f}\geq\lambda_{g,h,n}:=8-\frac{2(n+1)\bigl(g-1-n(h-1)\bigr)}{3\bar{a}_{1}}.
$$

{\rm (ii)} Assume that $\AL_{\SigmaT}(\ST)$ is a curve and $\nu(q_{\tilde{\theta}})\geq1$.
If  $F(g,h,\nu(q_{\tilde{\theta}}))\geq0$, then 
$$
\lambda_{f}\geq\lambda_{g,h,n,q_{\tilde{\theta}}}:=8-\frac{2(n+1)\bigl(g-1-n(h-1)\bigr)}{3\bar{a}_{\nu(q_{\tilde{\theta}})+1}}.
$$
\end{thm}

\begin{proof}
Here we restrict ourselves to the case that $\AL_{\SigmaT}(\ST)$ is a curve and show (ii) only,
since (i) can be shown similarly.
From (\ref{1.17}) and (\ref{1.18}), we obtain
\begin{align}
\nonumber
t(K_{f}^2-\lambda_{g,h,n,q_{\tilde{\theta}}} \chi_{f})& \geq t(x'-\lambda_{g,h,n,q_{\tilde{\theta}}}\bar{x}')\frac{K_{\PHI}^2}{(n-1)(h-1)}
+t(y'-\lambda_{g,h,n,q_{\tilde{\theta}}}\bar{y}')T \\
& +t(z'-\lambda_{g,h,n,q_{\tilde{\theta}}}\bar{z}')\AL_{0}
+\sum_{k\geq1}(a_{k}-\lambda_{g,h,n,q_{\tilde{\theta}}}\bar{a}_{k})\AL_{k}
-n\lambda_{g,h,n,q_{\tilde{\theta}}} t\chi_{\PHI}
\label{2.7}
\end{align}
To apply $(\ref{2.5})$ to the above inequality, we have to check that 
$z'-\lambda_{g,h,n,q_{\tilde{\theta}}} \bar{z}'\leq0$ in advance.
Since 
\begin{align}
\label{2.7''}
\lambda_{g,h,n,q_{\tilde{\theta}}}
&\geq8-\frac{2(n+1)\bigl(g-1-n(h-1)\bigr)}{3\bar{a}_{1}}\\
\nonumber
&=8-\frac{8\bigl(g-1-n(h-1)\bigr)}{(n-1)\bigl( 2(g-1)+n(h-1)(n-3)\bigr)},
\end{align}
it is sufficient to see that 
$$8-\frac{8\bigl(g-1-n(h-1)\bigr)}{(n-1)\bigl( 2(g-1)+n(h-1)(n-3)\bigr)}
\geq\frac{24(n-1)(g-1)}{2(2n-1)(g-1)-n(n+1)(h-1)},$$
which is equivalent to 
\begin{align*}
(g-1-n(h-1))\bigl(16(g-1)n(n-2)+8n(n+1)(h-1)((n-1)(n-3)+1)\bigr)\geq 0,
\end{align*}
the validity of which can be checked directly.
Therefore we get $z'-\lambda_{g,h,n,q_{\tilde{\theta}}}\bar{z}'\leq0$.
Applying $(\ref{2.5})$ to $(\ref{2.7})$, we obtain 
\begin{align*}
&t(K_{f}^2-\lambda_{g,h,n,q_{\tilde{\theta}}}\chi_{f}) \\
&\geq\frac{n-1}{12}t\biggl(12(g-1)-
\frac{3}{2}\lambda_{g,h,n,q_{\tilde{\theta}}}\bigl(g-1-n(h-1)\bigr)\biggr)
\frac{K_{\PHI}^2}{(n-1)(h-1)}-n\lambda_{g,h,n,q_{\tilde{\theta}}}t\chi_{\PHI}\\
&+\frac{(n-1)(8-\lambda_{g,h,n,q_{\tilde{\theta}}})}{8\bigl(g-1-n(h-1)\bigr)}tT\\
&+\frac{nt}{12}\sum_{k=1}^{\nu(q_{\tilde{\theta}})}
\big((n-1)k((2n-1)k-3)\lambda_{g,h,n,q_{\tilde{\theta}}}-12((n-1)k-1)^2\bigr)\AL_{k}\\
&+t\sum_{k>\nu(q_{\tilde{\theta}})}
\biggl( \frac{(8-\lambda_{g,h,n,q_{\tilde{\theta}}})3n}{2(n+1)\bigl(g-1-n(h-1)\bigr)}
\bar{a}_{k}-n\biggr)\AL_{k}.
\end{align*}
We will show that 
$\big((n-1)k((2n-1)k-3)\lambda_{g,h,n,q_{\tilde{\theta}}}-12((n-1)k-1)^2\bigr)
\geq0$ for $1\leq k \leq \nu(q_{\tilde{\theta}}).$
Note that 
\begin{align*}
&(n-1)k((2n-1)k-3)\lambda_{g,h,n,q_{\tilde{\theta}}}-12((n-1)k-1)^2\\
=&\frac{1}{n^2}\biggl(nk\bigl( (n-1)(nk-1)(\lambda_{g,h,n,q_{\tilde{\theta}}} (2n-1)-12(n-1))+
(n^2-1)(12-\lambda_{g,h,n,q_{\tilde{\theta}}}) \bigr)-12n^2\biggr).
\end{align*}
Firstly we will show that $\lambda_{g,h,n,q_{\tilde{\theta}}}
\geq \frac{12(n-1)}{2n-1}$. 
From $(\ref{2.7''})$, 
it is sufficient to check that 
$$
8-\frac{8\bigl(g-1-n(h-1)\bigr)}{(n-1)\bigl( 2(g-1)+n(h-1)(n-3)\bigr)}
\geq \frac{12(n-1)}{2n-1}.
$$
A calculation shows that it is equivalent to 
$$
8n(n-2)(g-1)+(4(n^2-1)(n-3)+8(2n-1))n(h-1)\geq 0,
$$
which holds true clearly.
So we have shown $\lambda_{g,h,n,q_{\tilde{\theta}}}
\geq \frac{12(n-1)}{2n-1}$ and it follows that  
$$
(n-1)k((2n-1)k-3)\lambda_{g,h,n,q_{\tilde{\theta}}}-12((n-1)k-1)^2
$$
is increasing in $k$. 
Evaluating at $k=1$, we get 
\begin{align*}
&(n-1)k((2n-1)k-3)\lambda_{g,h,n,q_{\tilde{\theta}}}-12((n-1)k-1)^2 \\
&\geq2(n-2)\bigl((n-1)\lambda_{g,h,n,q_{\tilde{\theta}}}-6(n-2)\bigr)\\
&\geq0,
\end{align*}
by $\lambda_{g,h,n,q_{\tilde{\theta}}}\geq \frac{12(n-1)}{2n-1}$.
Since
\begin{align}
\frac{(8-\lambda_{g,h,q_{\tilde{\theta}}}^1)3n}{2(n+1)\bigl(g-1-n(h-1)\bigr)}\bar{a}_{k}-n
\geq
\frac{(8-\lambda_{g,h,q_{\tilde{\theta}}}^1)3n}{2(n+1)\bigl(g-1-n(h-1)\bigr)}
\bar{a}_{\nu(q_{\tilde{\theta}})+1}-n=0
\label{2.10}
\end{align}
holds for any $k\geq\nu(q_{\tilde{\theta}})+1$, we obtain 
\begin{align}
\nonumber
& K_{f}^2-\lambda_{g,h,n,q_{\tilde{\theta}}}\chi_{f} \\
\geq & \;\frac{n-1}{12}\biggl(12(g-1)-
\frac{3}{2}\lambda_{g,h,n,q_{\tilde{\theta}}}\bigl(g-1-n(h-1)\bigr)\biggr)
\frac{K_{\PHI}^2}{(n-1)(h-1)} \label{2.12} \\
&-n\lambda_{g,h,n,q_{\tilde{\theta}}}\chi_{\PHI}
+\frac{(n-1)(8-\lambda_{g,h,n,q_{\tilde{\theta}}})}{8\bigl(g-1-n(h-1)\bigr)}T\nonumber
\end{align}

If $F(g,h,\nu(q_{\tilde{\theta}}))\geq0$ and $h=1$, 
then by $(\ref{1.19})$ we have 
$T=2(g-1)K_{\PHI}.R\geq \frac{4(g-1)^2}{n-1}\chi_{\PHI}$.
Hence it follows from $(\ref{2.12})$ that 
$$
K_{f}^2-\lambda_{g,1,n,q_{\tilde{\theta}}}\chi_{f}\geq
\frac{n+1}{3\bar{a}_{\nu(q_{\tilde{\theta}})+1}}F(g,1,\nu(q_{\tilde{\theta}}))\chi_{\PHI}\geq0
$$
which gives us (ii) for $h=1$.
If 
$F(g,h,\nu(q_{\tilde{\theta}}))\geq 0$ and $h\geq 2$, then we can use   
Xiao's slope inequality $K_{\PHI}^2 \geq \frac{4(h-1)}{h}\chi_{\PHI}$ and $T\geq 0$, to get 
$$
K_{f}^2-\lambda_{g,h,n,q_{\tilde{\theta}}}\chi_{f}
\geq \frac{n+1}{3h\bar{a}_{\nu(q_{\tilde{\theta}})+1}}
F(g,h,\nu(q_{\tilde{\theta}}))\chi_{\PHI}
\geq0
$$
from $(\ref{2.12})$.
Hence we have shown (ii) also for $h\geq 2$.
\end{proof}

\section{Special irregular cyclic covering fibrations of ruled surfaces.}

Let $f:S\to B$ be a primitive cyclic covering fibration of type $(g,0,n)$ with $q_f>0$ and 
suppose that it is not locally trivial.
Let $\alpha_{\SigmaT}:\ST \to \Alb_{\SigmaT}(\ST)$ be the morphism defined as in Definition 
\ref{(2.1)} for the generator $\SigmaT$ of the covering transformation group $G$ of $\tilde{\theta}:\widetilde{S}\to \widetilde{W}$.
Moreover we assume that there is a component $C$ of $\Fix(G)$ such that $C^2=0$.
Note then that $\alpha_{\SigmaT}(\ST)$ is a curve by Proposition \ref{prop2.2}.

\begin{prop}
\label{prop3.1}
In the above situation, there are a fibrations $\tilde{f'}:\ST \to B'$, $\PHI':\WT \to \mathbb{P}^1$ and 
a morphism $\theta':B' \to \mathbb{P}^1$, where $B'$ is s smooth curve, such that $\RT$ is $\tilde{\PHI'}$-vertical, 
$q_{f}\leq g(B')$, and they fit into the commutative diagram:
\begin{equation}
\nonumber
 \vcenter{ \xymatrix{ 
 \ST \ar[r]^{\tilde{\theta}} \ar[d]_{\tilde{f}'} 
 & \WT \ar[d]^{\tilde{\PHI}'} \\
 B' \ar[r]_{\theta'} 
 & \PROJ^1 . } } 
 \end{equation}
\end{prop}

\begin{proof}
We can obtain $\tilde{f'}:\ST \to B'$ from the Stein factorization of $\alpha_{\SigmaT}:\ST \to \alpha_{\SigmaT}(\ST)$. 
Hence we have $g(B')\geq q_f$ by Proposition~\ref{prop2.2}, (3).
We will show that the automorphism $\SigmaT:\ST \to \ST$ induces an automorphism of $B'$. 
We assume that there is a fiber $F'$ of $\tilde{f'}$ such that $\SigmaT^{\ast}F'$ has a $\tilde{f}'$-horizontal component. Let $F_{C}'$ be the fiber of $f'$ which contains the curve $C$ with $C^2=0$.
Then, from Zariski's lemma, we see that $F_{C}'=aC$
for some positive integer $a$ and it follows $F_{C}'=\SigmaT^{\ast}F_{C}'$, since $C$ is a component
of $\Fix(G)$.
Hence 
$$
0<(\SigmaT^{\ast}F'.F_{C}')=(\SigmaT^{\ast}F'.\SigmaT^{\ast}F_{C}')=(F'.F_{C}')=0,
$$
a contradiction. 
Therefore $\SigmaT$ maps fibers to fibers, and descends down to give an automorphism $\SigmaT_{B'}:B' \to B'$. 
Furthermore we have the commutative diagram
\begin{equation}
\nonumber
 \vcenter{ \xymatrix{ 
 \ST \ar[r]^{\tilde{\theta}} \ar[d]_{\tilde{f}'} 
 & \WT \ar[d]^{\tilde{\PHI}'} \ar[r]^{\tilde{\PHI}}& B\\
 B' \ar[r]_{\theta'} 
 & D' ,& } } 
 \end{equation}
where $\theta':B'\to D':=B'/\langle\SigmaT_{B'}\rangle$ denotes the quotient map.
In order to complete the proof, it suffices to see that $D'=\mathbb{P}^1$. 
This can be shown as follows. Any general fiber of $\tilde{\PHI}$ is  $\tilde{\PHI'}$-horizontal by $\RT.\GAT>0$. 
Since $\tilde{\PHI}$ is ruled, we see that $\PROJ^1$ dominates $D'$ and it follows $D'=\PROJ^1$.
\end{proof}

The contraction $\PHI:\WT \to W$ is composed of several blowing-ups. 
We decompose it as $\psi=\check{\psi}\circ\bar{\psi}$ as follows.
Let $\bar{\psi}:\WT \to \overline{W}$ be the longest succession of blowing-downs such that 
we still have the morphism $\bar{\PHI '}$ satisfying $\tilde{\PHI'}=\bar{\PHI'}\circ\bar{\psi}$. 
Then we have the following commutative diagram.
\[
\xymatrix{ 
 & \PROJ^1  &  \\
\WT \ar[ru]^{\tilde{\PHI}'}  \ar[r]^{\bar{\psi}} \ar[rd]_{\tilde{\PHI}} 
& \overline{W} \ar[d]^{\bar{\PHI}} \ar[r]^{\check{\psi}} \ar[u]^{\bar{\PHI}'}
& W \ar[ld]^{\PHI}\\
& B &
}
\]
Let $\overline{R}:=\bar{\psi}_{\ast}\RT$ be the image of $\RT$ by $\bar{\psi}$.

\begin{lem}
\label{prop3.3}
The morphism $\check{\psi}:\overline{W}\to W$ is not the identity map.
\end{lem}  

\begin{proof}
We will prove this by contradiction.
Suppose that $\check{\psi}$ is the identity map.
As one sees from the proof of Lemma~\ref{lem_added}, any irreducible curve $D$ on $W$ with $D^2\leq 0$ 
is smooth, and $\varphi|_D:D\to B$ is an unramified covering when $D^2=0$ and $D$ is not a fiber of $\varphi$.
Hence, any irreducible fiber of $\bar{\PHI'}:W\to \PROJ^1$ has to be smooth.
Suppose that there is a fiber of $\bar{\PHI'}$ whose reduced scheme is reducible, and take an irreducible component $D_0$.
Then $D_0^2<0$ and, from the proof of Lemma~\ref{lem_added}, we conclude that $D_0$ coincides with the minimal section.
The unicity implies that we cannot have such reducible singular fibers.
Therefore, a singular fiber of $\bar{\varphi'}$, if any, is a multiple fiber whose support is 
a smooth irreducible curve.
Since $R$ is a reduced divisor with support in fibers of $\bar{\varphi'}$ by Proposition~\ref{prop3.1}, 
we see that $R$ is smooth and $\PHI|_{R}:R\to B$ is unramified. 
Then $f:S\to B$ is a locally trivial fibration, which is inadequate. 
\end{proof}

Assume that $\theta':B' \to \mathbb{P}^1$ is branched over $\Delta \subset \mathbb{P}^1$. 
For any $y \in \Delta $, let $\widetilde{\Gamma_{y}'}=\sum \tilde{n}_{C}C$ be the fiber of $\tilde{\PHI'}$ over $y$, and put
$$
\RT_{all}:= \tilde{\PHI '}^{\ast}\Delta,\quad 
\RT_{r}:=\sum_{y\in \Delta,\;C\subset \widetilde{\Gamma_{y}'},\;\tilde{n}_{C}=1}C.
$$

\begin{lem}
In the above situation, $\RT_{r} \preceq \RT.$
\label{lem3.4}
\end{lem}

\begin{proof}
We put 
$$
G_{\tilde{f'}}:=\{\tau \in G \mid \tau(\widetilde{F}')=\widetilde{F}' \text{ for any fiber } \widetilde{F}'
\text{ of }\tilde{f'}\}
$$
Since $\FDT \circ \tau = \tau$ for any $\tau \in G_{\FDT}$, the morphism $\FDT$ induces the morphism $\pi:\ST/G_{\FDT} \to \WT$ and 
we have the following commutative diagram.
\begin{align}
\nonumber
\xymatrix{ 
\ST \ar[rd] \ar@/^18pt/[rrd]^{\tilde{\theta}} \ar@/_18pt/[ddr]_{\tilde{f}'}&   &  \\
& \ST/G_{\tilde{f'}} \ar[d] \ar[r]^{\pi}
& \WT \ar[d]^{\tilde{\PHI}'} \\
& B' \ar[r]_{\theta'} & \PROJ^1
}
\end{align}
Note that the degree of $\pi$ is equal to that of $\theta'$.
We claim that $\Fix(G_{\FDT})=\Fix(G)$.
This can be see as follows.
It is clear that  $\Fix(G_{\FDT})\supset\Fix(G)$.
If there is a point $x \in  \Fix(G_{\FDT})\setminus\Fix(G)$, 
then we have $\SigmaT(x)\neq x$ for the generator $\SigmaT$ of $G$.
On the other hand, since $G_{\FDT}$ is a subgroup of $G$ of order $n/\DEG\,\theta '$, we have $G_{\FDT}=\langle\SigmaT^{\DEG\,\theta '}\rangle$.
Hence the number of $G$-orbits of $x$ is at most $n/\DEG\,\theta '$.
This contradicts that $\tilde{\theta}: \ST \to \WT$ is totally ramified.
Therefore $\Fix(G_{\FDT})=\Fix(G)$. Hence $\ST/G_{\FDT}$ is smooth.
Let $R_{\pi}$ be the branch locus of $\pi$. 
Since $\tilde{\theta}$ is totally ramified, one can check easily that $R_{\pi}=\RT$.
Hence it is sufficient to prove that $\RT_{r}\leq R_{\pi}$.
Let $C$ be any component of $\RT_{r}$.
We can take analytic local coordinates $(U_\mathbb{P}^1,x)$ on $\mathbb{P}^1$, $(U_{\WT},y,z)$ on $\WT$ and $(U_{B'},w)$ on $B'$ such that $\tilde{\PHI '}(C)$ is defined by $x=0$, $C$ is defined by $y=0$, $\theta'^{\ast}x=w^{\DEG\,\theta'}$ and $\bar{\PHI '}^{\ast}x=y$.
$U_{B'}\times_{\mathbb{P}^1}U_{\WT}$ is defined by $y=w^{\DEG\,\theta'}$ in $U_{B'}\times U_{\WT}$.
So $ U_{B'}\times_{\mathbb{P}^1}U_{\WT} \to U_{\WT}$ is ramified over $C\cap U_{\WT}$ and $U_{B'}\times_{\mathbb{P}^1}U_{\WT} $ is smooth.
Hence the natural morphism $\ST/G_{\FDT} \to B'\times_{\mathbb{P}^1}\WT$ is an isomorphism around $U_{B'}\times_{\mathbb{P}^1}U_{\WT} $. 
Therefore we get $C\preceq R_{\pi}=\RT$.
\end{proof}

We suppose that $\check{\psi}=\check{\psi}_{1}\circ\;\cdots\;\circ\check{\psi}_{u}$, where $\check{\psi}_{i}:\check{W_{i}}\to\check{W}_{i-1}$ is 
a blowing-up at $\check{x}_{i-1}\in \check{W}_{i-1}$ with exceptional curve $\check{\mathcal{E}_{i}}\subset\WC$, $\WC_{0}=W$ and $\WC_{u}=\overline{W}$.
Let $\check{R}_{i}$ be the image of $\overline{R}$ in $\WC_{i}$, and let $\check{x_{i}}$ be a singular point of $\check{R_{i}}$ of multiplicity $\check{m_{i}}$.

\begin{lem}
\label{lemma3.5}
Assume that $n \geq 3$. For $1\leq i \leq u-1$, we have $\check{m}_{i} \geq \frac{2q_{f}}{n-1}+2$.
Moreover if there is $\check{m}_{i}$ such that equality sign holds, then $ \frac{2q_{f}}{n-1}+2 \in n\Z$ and 
$\deg \theta'=n$.
\end{lem}

\begin{proof}
Let $\mathcal{E}\subset\overline{W}$ be any $(-1)$-curve contracted by $\check{\psi}$.
Note that $\bar{\PHI'}|:\mathcal{E} \to \mathbb{P}^1$ is surjective and $\mathcal{E}.\overline{R} \in n\Z$. 
We will show that $\mathcal{E}.\overline{R} \geq \frac{2q_{f}}{n-1}+2$.
If $n\geq \frac{2q_f}{n-1}+2$, then the assertion is clear from Lemma~\ref{lem1.2}, (1).
Hence we may assume that $n<\frac{2q_f}{n-1}+2$ in the following.
In particular, we have $q_f\geq n-1$.

By Lemma $\ref{lem3.4}$, it is enough to show that  $\mathcal{E}.\overline{R}_{r} \geq \frac{2q_{f}}{n-1}+2$
where $\overline{R}_{r}$ is image of $\RT_{r}.$
Let $\mathcal{R}_{\theta'}$ be the ramification divisor of $\theta':B' \to \PROJ^1$.
Since $g(B')\geq q_f>0$, we have $\deg \theta'>1$.
From the Hurwitz formula, we get 
\begin{align}
\deg \mathcal{R}_{\theta'}=2g(B')-2+2\deg \theta'.
\label{(3.4)}
\end{align}
By Lemma $\ref{lem3.2}$ and $(\ref{(3.4)})$, we get $(\deg\theta'-1)\sharp \Delta \geq 2g(B')-2+2\deg \theta'$, that is, 
\begin{align*}
\sharp \Delta \geq \frac{2}{\deg \theta'-1}g(B')+2.
\end{align*}
We put $\overline{R}_{all}:=\bar{\psi'}^{\ast}\Delta$. For any $p\in\mathcal{E}\cap\overline{R}_{all}$, let $r_{p}:=I_{p}(\mathcal{E}.\overline{R}_{all})$ be the local intersection number.
Since $\overline{R}_{all}=\bar{\psi'}^{\ast}\Delta$ consists of $\sharp \Delta$ fibers of $\bar{\PHI'}$, one has 
\begin{align}
\nonumber
\sum_{\mathcal{E}\cap\overline{R}_{all}} r_{p}=(\mathcal{E}.\overline{R}_{all}) & =\deg(\overline{\PHI'}|_{\mathcal{E}})\sharp \Delta \\
& \geq (\deg \bar{\PHI'}|_{\mathcal{E}})\left(\frac{2}{\deg\theta'-1}g(B')+2\right).
\label{(3.5)}
\end{align}
By definition, $r_{p}\geq2$ for any  $p\in(\mathcal{E}\cap\overline{R}_{all})\setminus(\mathcal{E}\cap\overline{R}_{r})$.
On the other hand, by the Hurwitz formula for $\bar{\PHI'}|_{\mathcal{E}}:\mathcal{E} \to \PROJ^1$, one has 
$$
2\deg \bar{\PHI'}|_{\mathcal{E}}-2 = \deg \mathcal{R}_{\theta'} \geq \sum_{\mathcal{E}\cap\overline{R}_{all}} (r_{p}-1).
$$
Hence, 
\begin{align*}
2\deg \bar{\PHI'}|_{\mathcal{E}}-2&\geq\sum_{\mathcal{E}\cap\overline{R}_{all}} (r_{p}-1)
=\sum_{(\mathcal{E}\cap\overline{R}_{all})\setminus(\mathcal{E}\cap\overline{R}_{r})}(r_{p}-1) +\sum_{\mathcal{E}\cap\overline{R}_{r}}(r_{p}-1)\\
&\geq \sum_{(\mathcal{E}\cap\overline{R}_{all})\setminus(\mathcal{E}\cap\overline{R}_{r})}\frac{r_{p}}{2}
+\sum_{\mathcal{E}\cap\overline{R}_{r}}\frac{(r_{p}-1)}{2}
=\sum_{\mathcal{E}\cap\overline{R}_{all}}\frac{r_{p}}{2}-\frac{\sharp(\mathcal{E}\cap\overline{R}_{r})}{2}\\
&\geq (\deg \bar{\PHI'}|_{\mathcal{E}})\left(\frac{1}{\deg \theta'-1}g(B')+1\right)-\frac{\sharp(\mathcal{E}\cap\overline{R}_{r})}{2},
\end{align*}
where the last inequality comes from $(\ref{(3.5)})$.

Note that we have $g(B')\geq q_f\geq n-1\geq \deg \theta'-1$.
Therefore 
\begin{align*}
(\mathcal{E}.\overline{R})\geq(\mathcal{E}.\overline{R}_{r})\geq \sharp(\mathcal{E}\cap\overline{R}_{r})
& \geq (\deg \bar{\PHI'}|_{\mathcal{E}})\left(\frac{2}{\deg \theta'-1}g(B')-2\right)+4\\
&\geq\frac{2}{\deg \theta'-1}g(B')+2\\
&\geq\frac{2}{\deg \theta'-1}q_{f}+2.
\end{align*}
For any $\check{x_{i}}$, let $\check{x}_{i+j_{i}}$ be the last infinitely near singular point blown up by $\check{\psi}$, $\mathcal{E}_{i+j_{i}}$ the exceptional curve.
Then, from the above argument, we get 
$$
\frac{2}{\deg \theta'-1}q_{f}+2 \leq (\mathcal{E}_{i+j_{i}}.\overline{R})=\check{m}_{i+j_{i}}\leq \check{m_{i}}.
$$
Moreover if the equality signs hold everywhere, then we get
$$
\check{m_{i}}= \frac{2}{\deg \theta'-1}q_{f}+2= (\mathcal{E}_{i+j_{i}}.\overline{R})\in n\Z.
$$
Since $\deg \theta'\leq n$, we are done.
\end{proof}

\begin{thm}
\label{xiaoconj}
Let $f:S\to B$ be a locally non-trivial primitive cyclic covering fibration of type $(g,0,n)$ 
with $q_{f}>0$ and $n\geq 3$.
Assume that there is a component C of $\Fix(\SigmaT)$ such that $C^2=0$.
Then,
 \begin{align*}
 q_{f}\leq \frac{g-n+1}{2}.
 \end{align*}
\end{thm}

\begin{proof}
There is a singular point $x$ of $R$ which is blown up by $\check{\psi}:\overline{W} \to W$
from Lemma $\ref{prop3.3}$. If $m:=\mathrm{mult}_{x}R \leq \frac{r}{2}$, then 
\begin{align*} 
\frac{2q_{f}}{n-1}+2\leq m \leq \frac{r}{2}=\frac{g}{n-1}+1
\end{align*} 
from Lemma $\ref{lemma3.5}$. So we get $q_{f}\leq (g-n+1)/2$.
If $m>\frac{r}{2}$, then we have $m\in n\Z+1$ from Lemma~$\ref{lem1.4}$.
Let $\check{x}$ be the last singular point, infinitely near to $x$, blown up by $\check{\psi}$ 
and $\check{m}$ its multiplicity.
Then it holds that $\check{m} \in n\Z$.
Indeed, if $\check{m} \in n\Z+1$, the exceptional curve arizing from  
$\check{x}$ is contained in branch locus. 
It contradicts the definition of $\check{\psi}$.
So we get $\check{m}+1 \leq m$.
Therefore we get 
\begin{align*} 
\frac{2q_{f}}{n-1}+3\leq \check{m}+1 \leq m\leq \frac{r}{2}+1=\frac{g}{n-1}+2
\end{align*}
from Lemma $\ref{lem1.4}$ and Lemma $\ref{lemma3.5}$.
It follows $q_f\leq (g-n+1)/2$.
\end{proof}

Therefore, the Modified Xiao's Conjecture is true in this particular case.

\medskip

Now, we turn our attention to the slope.
Let $\check{\alpha}_{k}$ be the number of the singular points of $R$ with multiplicity $nk$ or $nk+1$ appearing in $\check{\psi}$.
Then $\check{\alpha}_{k}\geq0$ and by Lemma $\ref{lemma3.5}$ one has 
\begin{align}
\check{\alpha}_{k}=0
\label{(3.6)}
\end{align}
for any $k$ satisfying $nk+1 \leq \frac{2q_{f}}{n-1}+2$.
We put $\bar{\alpha}_{k}:=\AL_{k}-\ALC_{k}$ then by $(\ref{(3.6)})$,
\begin{align}
\numberwithin{equation}{section}
\ALB_{k}=\AL_{k}
\label{(3.7)}
\end{align}
for any $k$ satisfying $nk+1 \leq \frac{2q_{f}}{n-1}+2$.
By the construction of $\bar{\PHI'}$, $\overline{R}$ is contained in fibers of  $\bar{\PHI'}$,
hence we get
$$\overline{R}^2\leq0.$$
On the other hand, we have
$$\overline{R}^2=R^2-\sum_{\frac{2q_{f}}{n-1}+2\leq nk}^{nk\leq\frac{r}{2}}n^2 k^2 \ALC_{k}$$
As $R^2=2rM_{0}$ by ($\ref{(1.15)}$), we get 
\begin{align}
\numberwithin{equation}{section}
2rM_{0}\leq\sum_{\frac{2q_{f}}{n-1}+2\leq nk}^{nk\leq\frac{r}{2}}n^2 k^2 \ALC_{k}
\label{(3.8)}
\end{align}
Hence 
\begin{align*}
\AL_{0}+\sum_{k\geq1}^{nk+1 \leq \frac{2q_{f}}{n-1}+2}nk(nk-1)\AL_{k}&\leq\AL_{0}+ \sum_{k\geq 1}^{nk\leq\frac{r}{2}}nk(nk-1)\bar{\alpha}_{k}\\
&\leq\sum_{\frac{2q_{f}}{n-1}+2\leq nk}^{nk\leq\frac{r}{2}}\frac{nk}{r}(r-1-(nk-1))\ALC_{k}\\
&\leq\sum_{\frac{2q_{f}}{n-1}+2\leq nk}^{nk\leq\frac{r}{2}}\frac{nk}{r}(r-1-(nk-1))\AL_{k},
\end{align*}
where the first and the last inequalities above follow immediately from $\ALB_{k}\geq0$ and $(\ref{(3.7)})$, 
and the second one follows from ($\ref{(1.16)})$ and ($\ref{(3.8)})$.
Hence we have shown:

\begin{prop}
\label{prop3.7}
Under the same assumptions as in Proposition $\ref{prop3.1}$,  
$$
\AL_{0}+\sum_{k\geq1}^{nk+1 \leq \frac{2q_{f}}{n-1}+2}nk(nk-1)\AL_{k}\leq\sum_{\frac{2q_{f}}{n-1}+2\leq nk}^{nk\leq\frac{r}{2}}\frac{nk}{r}(r-1-(nk-1))\AL_{k}.
$$
\end{prop}


Using this, we will prove the following:

\begin{thm}
Let $f:S \to B$ be a locally non-trivial primitive cyclic covering fibration of type $(g,0,n)$ such that 
there is a component $C \subset \Fix(G)$ with $C^2 = 0$. 
If $q_f>0$ and $n\geq 3$, then
\begin{align}
\numberwithin{equation}{section}
\lambda_{f}\geq\lambda_{g,n,q_{f}}^{2}:=8-\frac{2n(g+n-1)}{(g-q_{f})(q_{f}+n-1)}
\biggl(=8-\frac{4rn}{(n-1)(2+\frac{2q_{f}}{n-1})(r-(2+\frac{2q_{f}}{n-1}))}\biggr).
\label{(3.9)}
\end{align}
\label{thm4.7}
\end{thm}

\begin{proof}
We first remark that, for two real numbers $x$, $y$ with $x+y\leq r$, we have $x(r-x)\geq y(r-y)$ if and only if $x\geq y$.
Since we have $n+(2+2q_f/(n-1))\leq r$ by the proof of Theorem~\ref{xiaoconj} and $r\geq 2n$, this observation works for 
$x=n$, $y=2+2q_f/(n-1)$.

\smallskip

(i) The case of $n\geq 2+\frac{2q_{f}}{n-1}$, i.e., $\frac{(n-2)(n-1)}{2}\geq q_{f}$.

Since $n\geq 2+\frac{2q_{f}}{n-1}$, we get 
$$
n(r-n)\geq\bigl( 2+\frac{2q_{f}}{n-1} \bigr) \bigl( r-(2+\frac{2q_{f}}{n-1}) \bigr).
$$
Therefore, $\lambda_{g,n}^1 \geq \lambda_{g,n,q_{f}}^2$, 
and  $(\ref{(3.9)})$ follows from Theorem $\ref{prop2.3}$. 

\smallskip

(ii) The case of $n< 2+\frac{2q_{f}}{n-1}$. 

In this case, we have 
\begin{align*}
\bigl( 2+\frac{2q_{f}}{n-1} \bigr) \bigl( r-(2+\frac{2q_{f}}{n-1}) \bigr)\geq n(r-n)
\end{align*}
and, hence, $\lambda_{g,n}^1 \leq \lambda_{g,n, q_f}^2$. 
Then we have $A(\lambda_{g,n,q_{f}}^2)\leq0$, since the function 
$A(\lambda)$ defined in the proof of Proposition~\ref{prop2.3} is decreasing in $\lambda$ and we have already proved $A(\lambda_{g,n}^1)\leq0$ there. 

From Proposition~\ref{prop1.2}, we have 
\begin{align}
\numberwithin{equation}{section}
(r-1)(K_{f}^2-\lambda_{g,n,q_{f}}^2\chi_{f})\geq A(\lambda_{g,n,q_{f}}^2)\AL_{0}+
\sum_{k\geq1}^{nk \leq \frac{r}{2}}(a_{k}-\lambda_{g,n,q_{f}}^2\bar{a}_{k})\alpha_{k}
\label{(3.10)}
\end{align}
where $a_{k}$ and $\bar{a_{k}}$ are the same as in the proof of Proposition~\ref{prop2.3}.
Applying Proposition~\ref{prop3.7} to (\ref{(3.10)}), we get 
\begin{align*}
(r-1)(K_{f}^2-\lambda_{g,n,q_{f}}^2\chi_{f})&\geq\sum_{k\geq 1}^{nk< \frac{2q_f}{n-1}+2}
(-A(\lambda_{g,n,q_{f}}^2)nk(nk-1)+
a_{k}-\lambda_{g,n,q_{f}}^2\bar{a}_{k})\alpha_{k}\\
&+\sum_{2+\frac{2q_{f}}{n-1}\leq nk}^{nk \leq \frac{r}{2}}(A(\lambda_{g,n,q_{f}}^2)\frac{nk}{r}(r-nk)+a_{k}-\lambda_{g,n,q_{f}}^2\bar{a}_{k})\alpha_{k}.
\end{align*}
First, we will show 
\begin{align}
-A(\lambda_{g,n,q_{f}}^2)nk(nk-1)+
a_{k}-\lambda_{g,n,q_{f}}^2\bar{a}_{k}\geq0
\label{3.12'}
\end{align}
for any positive integer $k$ with $nk+1\leq2+\frac{2q_{f}}{n-1}$.
By a simple calculation, we get 
\begin{align}
\nonumber 
&-A(\lambda_{g,n,q_{f}}^2)nk(nk-1)\AL_{0}+
a_{k}-\lambda_{g,n,q_{f}}^2\bar{a}_{k}\\
=&nk\biggl( \frac{(n-1)(r-1)}{12n}(nk-1)(\lambda_{g,n,q_{f}}^2 (2n-1)-12(n-1))\label{(3.12)} \\
&\hspace{4cm} +\frac{(n^2-1)(r-1)}{12n}(12-\lambda_{g,n,q_{f}}^2)\biggr)-(r-1)n. \nonumber
\end{align} 
We claim that $\lambda_{g,n,q_{f}}^2\geq\frac{12(n-1)}{2n-1}$.
It is equivalent to 
$$
4(n+1)q_{f}(g-n+1-q_f)+(2n-4)g-2n(n-1)(2n-1)\geq 0.
$$
From $g-n+1\geq 2q_{f}$ and $q_f\geq \frac{(n-2)(n-1)}{2}+1$, we easily see that it holds true.
Since $\lambda_{g,n,q_{f}}^2\geq \frac{12(n-1)}{2n-1}$, the right hand side of (\ref{(3.12)}), which is incleasing in $k$, is not less than 
\begin{align*}
 &n\biggl( \frac{(n-1)(r-1)}{12n}(n-1)(\lambda_{g,n,q_{f}}^2 (2n-1)-12(n-1))+
\frac{(n^2-1)(r-1)}{12n}(12-\lambda_{g,n,q_{f}}^2) \biggr)-(r-1)n\\
&=\frac{n(n-1)(r-1)}{6}\biggl(\lambda_{g,n,q_{f}}^2(n-2)-6(n-3)\biggr)-(r-1)n\\
&= n(r-1)\frac{n(n-2)}{2n-1}\\
&\geq0.
\end{align*}
Therefore we get (\ref{3.12'}).

Secondly, we will show 
\begin{align}
A(\lambda_{g,n,q_{f}}^2)\frac{nk}{r}(r-nk)+a_{k}-\lambda_{g,n,q_{f}}^2\bar{a}_{k}\geq0
\label{3.14'}
\end{align}
for any positive integer $k$ satisfying $\frac{r}{2}\geq nk\geq 2+\frac{2q_{f}}{n-1}$.
By a simple calculation, we get
\begin{align}
\nonumber
&A(\lambda_{g,n,q_{f}}^2)\frac{nk}{r}(r-nk)+a_{k}-\lambda_{g,n,q_{f}}^2\bar{a}_{k}\\
=&\frac{(n-1)(r-1)}{4r}nk(r-nk)(8-\lambda_{g,n,q_{f}}^2)-(r-1)n
\label{(3.14)}
\end{align} 
Since we have 
$$nk(r-nk)\geq (2+\frac{2q_{f}}{n-1})\bigl( r-(2+\frac{2q_{f}}{n-1}) \bigr)$$
for any positive integer $k$ satisfying $2+\frac{2q_{f}}{n-1}\leq nk \leq \frac{r}{2}$, 
the right hand side of $(\ref{(3.14)})$ is not less than 
\begin{align*}
\frac{(n-1)(r-1)}{4r}(2+\frac{2q_{f}}{n-1})\bigl( r-(2+\frac{2q_{f}}{n-1}) \bigr)
(8-\lambda_{g,n,q_{f}}^2)-(r-1)n=0.
\end{align*}
In sum, we have shown $K_{f}^2-\lambda_{g,n,q_{f}}^2\chi_{f}\geq0$.
\end{proof}

\section{An example.}
We construct primitive cyclic covering fibrations of type $(g,0,n)$ with 
relative minimal irregularity $q_{f}$ satisfying $g+n-1=m(q_{f}+n-1)$ for any integer $m \geq 2$.
Hence, when $m=2$, this implies that the bound of $q_f$ in Theorem~\ref{xiaoconj} is sharp.
Also, our examples show that the slope bound (\ref{(3.9)}) in Theorem~\ref{thm4.7} is sharp.

Let $\PHI:W:=\PROJ(\mathcal{O}_{\PROJ^{1}}\bigoplus \mathcal{O}_{\PROJ^{1}}(e))\to B:=\PROJ^1$ be the Hirzebruch surface of degree $e \geq 0$.
Denote by $\GA$ and $C_{0}$ a fiber of $\PHI$ and
the section with $C_{0}^2=-e$, respectively.
We know that $mC_{0}+b\GA$ is very ample if and only if $b>me$.
So we take $b_{0}$ with $b_{0}>me$.

We take two general members $D,D'$ of $|mC_{0}+b_{0}\GA|$ 
which intersect each other transversely.
Let $\Lambda$ be the pencil generated by $D$ and $D'$.
Then $\Lambda$ define the rational map $\PHI_{\Lambda}:W \cdots\to \PROJ^1$.
Let $\psi$ be a minimal succession blowing-ups which eliminates the base points of $\Lambda$. 
We get a relatively minimal fibration $\tilde{\PHI '}:\WT \to\PROJ^1$ by putting $\tilde{\PHI '}=\PHI_{\Lambda} \circ \psi$.
Denote by $\GAT'$ a general fiber of $\tilde{\PHI}'$ and $K_{\WT}$ 
a canonical divisor of $\WT$.
By a simple calculation, we get 
\begin{align}
\numberwithin{equation}{section}
K_{\WT}^2=8-x,\;K_{\WT}.\GAT'=\frac{m-1}{m}x-2m,\;\GAT^2=0,
\label{5.1}
\end{align}
where $x$ is the a number of blowing-ups in $\psi$. 
Note that $x=(mC_{0}+b_{0}\GA)^2$.

Let $\Delta \subset \PROJ^1$ be a set of $\frac{2q}{n-1}+2$ general points, where 
$q$ is an integer satisfying $\frac{2q}{n-1}+2\in n\Z$.
Then there is a divisor $\D'$ on $\PROJ^1$ such that $n\D'=\Delta$.
Let $\RT=(\tilde{\PHI}')^{\ast}\Delta$ be the fiber of $\tilde{\PHI}'$ over $\Delta$.
Since $\Delta$ is general, we can assume that $\RT$ is both reduced and smooth.

We consider a classical cyclic $n$-covering
$$\theta':B'=
\mathrm{Spec}_{\PROJ^1}\biggl(\bigoplus_{j=0}^{n-1} \mathcal{O}_{\PROJ^1}(-j\D')\biggr)\to\PROJ^1 .$$
Since $\Delta$ is general, we can assume that 
the fiber product $\ST:=B\times_{\PROJ^1}\WT$ is smooth.
Noting that the morphism $\tilde{\theta}:\ST \to \WT$ induced by $\theta'$ is nothing but the natural one
$$
\ST=\mathrm{Spec}_{\PROJ^1}\biggl(\bigoplus_{j=0}^{n-1} 
\mathcal{O}_{\WT}(-j(\PHI')^{\ast}\D')\biggr)\to \WT,
$$
one gets a commutative diagram 
\begin{equation}
\nonumber
 \vcenter{ \xymatrix{ 
  B & \\
 \ST \ar[u]^{\tilde{f}} \ar[r]^{\tilde{\theta}} \ar[d]_{\tilde{f}'} &  \WT \ar[lu]_{\tilde{\PHI}} \ar[d]^{\tilde{\PHI}'} \\
 B' \ar[r]_{\theta'}   & \PROJ^1  } } 
 \end{equation}
where $\tilde{f}:=\tilde{\PHI} \circ \tilde{\theta}$.

By the construction, we get $q_f=q=g(B')$.
From the formulae
\begin{align*}
 K_{\ST}^2=n(K_{\WT}+\frac{n-1}{n}\RT)^2,\quad 
 \chi(\mathcal{O}_{\ST})=
 n\chi(\mathcal{O}_{\WT})+\frac{1}{2}\sum_{j=1}^{n-1}\frac{1}{n}j\RT(\frac{1}{n}j\RT+K_{\WT}),
\end{align*}
and $(\ref{5.1})$, we get
\begin{align}
\numberwithin{equation}{section}
 K_{\ST}^2=\left(4\frac{m-1}{m}\left(\frac{q}{n-1}+1\right)(n-1)-n\right)x+8\left(n-m(n-1)\left(\frac{q}{n-1}+1\right)\right),
\label{5.2}
\end{align}
\begin{align}
\numberwithin{equation}{section}
 \chi(\mathcal{O}_{\ST})=\frac{m-1}{2m}\left(\frac{q}{n-1}+1\right)(n-1)x+n
 -m(n-1)\left(\frac{q}{n-1}+1\right).
\label{5.3}
\end{align}
Let $g$ be the genus of fibration $f:\ST \to B$.
Then it is easy to see that 
\begin{align*}
\frac{2g}{n-1}+2=m\left(\frac{2q}{n-1}+2\right).
\end{align*}
Hence we get 
\begin{align*}
 &K_{\tilde{f}}^2=K_{\ST}^2-8(g-1)(g(B)-1)=\left(4\frac{m-1}{m}\left(\frac{q}{n-1}+1\right)(n-1)-n\right)x,\\
 &\chi_{\tilde{f}}=\chi(\mathcal{O}_{\ST})-(g-1)(g(B)-1)=
\frac{m-1}{2m}\left(\frac{q}{n-1}+1\right)(n-1)x.
\end{align*}
Therefore we get 
\begin{align*}
\lambda_{\tilde{f}}=\frac{K_{\tilde{f}}^2}{\chi_{\tilde{f}}}=
8-\frac{2n(g+n-1)}{(g-q_{f})(q_{f}+n-1)}
\end{align*}
by $q=q_{f}$. 

We remark that $\tilde{f}$ is relatively minimal. In fact the singular points of $R$, 
the image of $\RT$ in W, are all of multiplicity $\frac{2q_{f}}{n-1}+2\in n\Z$
and can be resolved by a single blowing-up.
So there is no $\tilde{\PHI}$-vertical $(-n)$-curve in $\WT$.
Therefore there is no $\tilde{f}$-vertical $(-1)$-curve in $\ST$.

{}

\bigskip

Department of Mathematics, Graduate School of Science, Osaka University,

1--1 Machikaneyama, Toyonaka, Osaka 560-0043, Japan

e-mail address: u802629d@ecs.osaka-u.ac.jp
 \end{document}